\documentclass{amsart}

\usepackage{graphicx}
\usepackage{epsfig}
\usepackage{epstopdf}
\usepackage{psfrag}

\usepackage{amssymb}
\usepackage{amsmath}
\usepackage{amsfonts}
\usepackage{amsthm}

\usepackage{enumitem}
\usepackage{color}
\usepackage{setspace}

\usepackage[mathscr]{eucal}

\usepackage{layout}
\usepackage{multicol}

\usepackage{hyperref}

\usepackage{fancyhdr}

\input xy
\xyoption{all}
\xyoption{arc}
\xyoption{tile}
\xyoption{import}
\xyoption{2cell}
\xyoption{poly}
\xyoption{web}
\xyoption{knot}

\newcommand{\ga}{\ensuremath{\alpha}}
\newcommand{\gb}{\ensuremath{\beta}}

\newcommand{\gd}{\ensuremath{\delta}}
\renewcommand{\ge}{\ensuremath{\varepsilon}}

\newcommand{\gf}{\ensuremath{\varphi}}
\renewcommand{\gg}{\ensuremath{\gamma}}
\newcommand{\gh}{\ensuremath{\eta}}

\newcommand{\gl}{\ensuremath{\lambda}}
\newcommand{\gm}{\ensuremath{\mu}}

\newcommand{\gs}{\ensuremath{\sigma}}
\newcommand{\gt}{\ensuremath{\tau}}

\newcommand{\gy}{\ensuremath{\psi}}

\newcommand{\C}{\mathbb{C}}
\newcommand{\Q}{\mathbb{Q}}
\newcommand{\R}{\mathbb{R}}

\newcommand{\Z}{\mathbb{Z}}

\newcommand{\lto}{\longrightarrow}

\newcommand{\mto}{\mapsto}

\newcommand{\co}{\colon\,}
\newcommand{\st}{ \: | \: }

\newcommand{\codim}{\text{codim}}

\newcommand{\gD}{\ensuremath{\Delta}}
\newcommand{\gF}{\ensuremath{\Phi}}
\newcommand{\gG}{\ensuremath{\Gamma}}
\newcommand{\gL}{\ensuremath{\Lambda}}

\newcommand{\mc}[1]{\ensuremath{\mathcal{#1}}}
\newcommand{\msc}[1]{\ensuremath{\mathscr{#1}}}

\newcommand{\bs}{\backslash}

\usepackage{tikz}

\usepackage{amssymb}

\usepackage{graphicx}

\usepackage{cypriot}

\usepackage{color}

\newtheorem{theorem}{Theorem}[section]
\newtheorem{lemma}[theorem]{Lemma}

\theoremstyle{definition}
\newtheorem{definition}[theorem]{Definition}

\newtheorem{proposition}[theorem]{Proposition}
\newtheorem{corollary}[theorem]{Corollary}

\theoremstyle{remark}
\newtheorem{remark}[]{Remark}
\newtheorem*{notation}{Notation}

\theoremstyle{problem}

\begin{document}

\title[Limits with braid arrangements]{Limits with braid arrangements}


%    author one information
\author{Matthew S. Miller}
\address{Department of Mathematics, Vassar College, Poughkeepsie, NY 12604,USA}
\curraddr{}
\email{mamiller@vassar.edu}

\author{Max Wakefield}
\address{Department of Mathematics, United States Naval Academy, Annapolis, MD, 21402 USA}
\curraddr{}
\email{wakefiel@usna.edu}

%    \subjclass is required.
\subjclass[]{}

\date{}

\dedicatory{}

%    Abstract is required.
\begin{abstract}
We define a monoid structure on the set of $k$-equal arrangements and use this structure to define limits of braid arrangements.  We compute the cohomology of the associated limits of rational models of the arrangements complex complements. We collect these complexes together into one complex by creating a new differential and product on their direct sum and show that the resulting complex is exact. 
\end{abstract}

\maketitle


%%%%%%%%%%%%%%%%%%%%%%%%%%%%%%%%%%%%%%%
%%%%%%%%%%%%%%%%%%%%%%%%%%%%%%%%%%%%%%%

\section{Introduction}

Artin's braid groups appear in numerous fields of mathematics as well as in applications to many different areas of science. The recent and well written books \cite{KT} by Kassel and Turaev and \cite{BCH} edited by Berrick, Cohen and Hanbury have popularized the study of braids to a wide audience. There are many different ways to define the braid groups. We focus on the perspective of the braid group as the fundamental group of the configuration space $\mathrm{Conf} (\ell,\R^2)$ of $\ell$ points in $\R^2$ (which has an interesting history see \cite{Mag-74}). In this note, we view the configuration space $\mathrm{Conf} (\ell ,\R^2)$ as the complement of the following arrangement of complex hyperplanes in $\C^\ell$ $$\mc{A}_\ell =\left\{ H_{ij}|1\leq i<j\leq \ell \right\}$$ modulo the symmetric action on the coordinates where $H_{ij}=\mathrm{er} \{x_i-x_j\}$. For this reason $\mc{A}_\ell$ is called the braid arrangement and has arisen in many other contexts including the Type A reflection groups (see \cite{OT}) and even as in a generalized proof of Arrow's Impossibility Theorem in statistical economics (see \cite{T-arrow}).

The primary aim of this study is to understand some limiting behavior of the cohomology of the complex complements of the braid arrangements. In order to do this we use certain subspace arrangements that are subarrangements of the braid arrangement called $k$-equal arrangements defined explicitly in Section \ref{prelim} below. These arrangements can also be related to configuration spaces as the configurations of $\ell$ points where up to $k-1$ of them are allowed to collide. 

The $k$-equal arrangements also have an interesting history. They were first defined in the work of Bj\"orner, Lov\'asz, and Yao in \cite{BLY-k-equal} on the computational complexity of linear decision trees. Then in \cite{BW-95} Bj\"orner and Welker studied these arrangements exclusively, particularly the cohomology of their real and complex complements, by using the monumental work of Goresky and MacPherson in \cite{GM88} on stratified Morse theory. In \cite{BW-95}, they develop very complicated recursive formulas for computing the Betti numbers of these arrangement complements. While these formulas are brilliant, they are very difficult to compute with and are not interpreted in simple combinatorial terms, both of which are goals in understanding arrangements. 

We undertook this work as a part of a larger study of the structure of the cohomology and rational formality of the complex complements of $k$-equal arrangements.  Our plan is to use these results to prove further theorems about $k$-equal arrangements.  Since \cite{BLY-k-equal} and \cite{BW-95}, many others have studied $k$-equal arrangements. In \cite{Yuz-SmalRat}, Yuzvinsky further develops the rational models created in \cite{DCP95} by De Concini and Procesi which are quasi-isomorphic to Morgan's rational models in \cite{M-78}. Then in \cite{Yuz-Small}, Yuzvinsky uses these models to present a new basis for the cohomology of the complex complement. The works by Bj\"orner and Wachs in \cite{BW-96}, de Longueville and Schultz in \cite{dLS}, Peeva, Reinner and Welker in \cite{PRW-99}, and Ziegler and \v{Z}ivaljevi\'c in \cite{ZZ93} have further refined our understanding of the Betti numbers and algebraic structure of the cohomology of $k$-equal arrangements. However, much is still unknown.

Another focus in the study of the topology of subspace arrangements is the question of rational formality. A direct consequence of Brieskorn's landmark result in \cite{B-73} is the formality of the complex complement of hyperplane arrangements.  More recently Feitchner and Yuzvinsky in \cite{FY-Formal} show that any subspace arrangement with a geometric intersection lattice has a rationally formal complex complement. Then Matei and Suciu in \cite{MS-00} found non-trivial Massey products in some complicated arrangements of real planes in $\R^4$, hence giving the first non-formal arrangement complement. Using the work of Baskakov in \cite{Baskakov-03} on moment angle complexes, Denham and Suciu in \cite{DS-Massey} demonstrate non-trivial Massey products on complements of subspace arrangements consisting entirely of coordinate subspaces. 

Then the authors in \cite{MW-Edge} presented a simple set of combinatorial criteria for non-trivial Massey products in subarrangements of the braid arrangement.  In \cite{MW-Pascal} the authors explicitly compute the cohomology rings a family of subspace arrangements of the braid arrangement whose intersection lattices are not geometric yet they support formal complex complements (other examples can be found in \cite{FY-Formal}). Most recently in \cite{M-k-equal} the first author demonstrates non-trivial Massey products in the $k$-equal arrangements with $k=3$ and $\ell \geq 7$.

The focus of this study is to define a limit of $k$-equal arrangements and their associated relative atomic complexes, collect these complexes together and uncover their algebraic and combinatorial structure. There are two main results. The first, Theorem \ref{homology}, presents a generating set for the homology of the limiting complexes which turns out to be much smaller than that in \cite{Yuz-SmalRat}. The second, Theorem \ref{decomp}, shows that under another differential the direct sum of these limiting complexes is exact. 

The flow of this paper is as follows: In Section \ref{prelim} we examine the relative atomic complexes specifically for $k$-equal arrangements and develop some notation. In Section \ref{arrts} we define a monoid structure on the set of $k$-equal arrangements and use this monoid structure to construct limits of the braid arrangements. Then in Section \ref{alg} we study the limits of the relative atomic complexes, compute homology of these complexes, and develop basic degree and algebra structures. In Section \ref{Bi} we sum all the limiting complexes together, define a new differential on this total complex, and the compute its homology. Finally, in Section \ref{gproduct} we define a graded product on the sum of all the limiting complexes and describe its connection to the previous differentials.

\begin{flushleft} {\bf Acknowledgments:} The authors are grateful to John McCleary for his inspiration that led to this paper and many useful discussions. The second author has been supported by the Simons Foundation and the NSF.\end{flushleft}

%%%%%%%%%%%%%%%%%%%%%%%%%%%%%%%%%%%%%%%
%%%%%%%%%%%%%%%%%%%%%%%%%%%%%%%%%%%%%%%

\section{Relative Atomic Complexes of $k$-equal arrangements}\label{prelim}

Let $V$ be an $\ell$ dimensional complex vector space. 

\begin{definition}
Let $\gs \subset [\ell] = \{1,2, \ldots, \ell\}$.  Define the subspace $X(\gs)$ as
$$X(\gs) = \{ (v_1 , \ldots , v_\ell) \in V \, | \, v_{i} = v_{j} \text{ for all } i,j \in \gs \}.$$
For a fixed $k$ with $\ell \geq k \geq 2$, we denote by \msc{S} the set $\{ \gs \subset [\ell] \, | \, |\gs|=k\}$.  The {\em $k$-equal arrangement}, $\mc{A}_{k,\ell}$ is the collection
$$ \mc{A}_{k,\ell} =\{ X(\gs) \, | \, \gs \in\msc{S} \}.$$
The complement of $\mc{A}_{k,\ell}$ is denoted by $M_{k,\ell}$, is defined by $M_{k,\ell} = V \setminus \left( \bigcup_{\gs \in \msc{S}} X(\gs) \right)$ and is called a {\em $k$-equal manifold}.
\end{definition}

Notice that in this notation $\mc{A}_{2,\ell}=\mc{A}_\ell$ is the braid arrangement.

\begin{notation}
We associate the set $\gs$ with the subspace $X(\gs)$. Then we use $S$ to denote an ordered subset of atoms in the labeled intersection lattice of $\mc{A}_{k,\ell}$, so
$$S =\{ \gs_1<\gs_2 \ldots < \gs_r\} \subseteq \msc{S}.$$
We often consider the intersection of a collection of subspaces associated to such a subset and so introduce the notation $X(S)$ to denote $\bigcap_{\gs \in S} X(\gs)$.
\end{notation}

\begin{definition}
Fix $\ell$ and $k$.  Let $S= \{ \gs_1, \ldots \gs_r \}$ and let $a_S \in A_{\ell-k}$.  Let $\gL = \{\gl_1 , \ldots , \gl_r\}$ where $\gl_j \subset [\ell]$ with the properties $\gl_j \cap \gs_j$ is empty and $\gl_j \cup \gs_j = [\ell]$.  We call $\gl_j$ the complement of $\gs_j$ and write $\gl_j = \gs_j^c$.  Further $\gL$ the complement of $S$ and write $\gL = S^c$.
\end{definition}

\begin{notation} 
In addition to the notation $X(S)$, it is convenient to use $\gL = S^c$ to denote the same subspace.  We therefore set $Y(\gL) = X(S)$.
\end{notation}

We make frequent use of Yuzvinsky's relative atomic complex, \cite{Yuz-SmalRat}, a rational homotopy model for the complement of a complex subspace arrangement.
\begin{definition} Let $A_{k,\ell}$ be the differential graded algebra over $\Q$ with a generator $a_{S}$ with degree
\begin{align}\label{deg}
\text{deg} (a_S) = 2 \, \text{codim} (X(S)) - | S |
\end{align}
 for each subset $S \subset \msc{S}$.  The differential in $D_{k,\ell}$ is defined by
\begin{gather} \label{d}
d a_{S} = \sum (-1)^j a_{S \bs \gs_j},
\end{gather}
where the sum is over $\gs_j \in S$ with $X(S \bs \gs_j) = X(S)$.  The product structure of $A_{k,\ell}$ is defined by
\begin{gather} \label{product}
a_{S} a_{T} =
\begin{cases}
(-1)^{\ge(S,T)} a_{S \cup T} & \codim \big(X(S)\big) + \codim \big(X(T)\big) = \codim \big(X(S \cup T)\big) \\
0 & \text{otherwise},
\end{cases}
\end{gather}
where $\ge(S,T)$ is the sign associated with the permutation that re-orders $S \cup T$ so that the elements of $T$ come after that of $S$. The algebra $A_{k,\ell}$ is the {\em relative atomic complex for the arrangement $\mc{A}_{k,\ell}$}.
\end{definition}

\begin{notation}
Again it is useful to use $\gL = S^c$ to denote elements of $A_{k,\ell}$ and set $b_\gL = a_S$. For a set of sets $S$, we write $\cap S$ for $\bigcap\limits_{s \in S} s$  and $\cup S$ for $\bigcap\limits_{s \in S} s$.
\end{notation}

Throughout this paper we use the notation setup in this section.  In particular, we attempt to use the symbols defined here only to represent the quantities they describe in this section.

%%%%%%%%%%%%%%%%%%%%%%%%%%%%%%%%%%%%%%%
%%%%%%%%%%%%%%%%%%%%%%%%%%%%%%%%%%%%%%%

\section{Limits of arrangements}\label{arrts}

The standard direct product of complex vector spaces,

\begin{align}
\C^\ell \times \C^m & \to \C^{\ell+m}\\
(x_1 , \ldots , x_\ell) \times (y_1, \ldots , y_m) & \to (x_1 , \ldots , x_\ell , y_1 , \ldots , y_m)\\
\intertext{give rise to a monoid structure on the set $\{\mc{A}_{k,\ell} \, | \, 0 < k \leq \ell \}$ defined by}
\mc{A}_{k,\ell} \times \mc{A}_{t,m} & \to \mc{A}_{k+t,\ell+m} \\
X(\gs) \times X(\gt) & \to X(\gs \cup \gt^{+m}) 
\end{align}

where for $\gt = \{ j_1 , \ldots , j_t\}$, we use the notation $\gt^{+m}$ to denote the set $\{ j_1+m , \ldots , j_t + m \}$.

In particular, if we take $\mc{A}_{1,1}$ to be $\{\C\}$, with an element of $\C$ considered to be ``equal", there is a multiplication on the right by $\mc{A}_{1,1}$ which takes $\vec x$ to $(\vec x, y)$ and maps $X(\gs)$ to $X(\gs \cup \{ \ell +1\})$.  Using the complement notation this multiplication maps $Y(\gl) \in \mc{A}_{k, \ell}$ to $Y(\gl) \in \mc{A}_{k+1,\ell+1}$.  So $Y(\gL)$ makes sense in any $\mc{A}_{k,\ell}$ provided that $\ell$ is greater than or equal to the maximum element in $\cup_{\gl \in \gL} \gl$.

%\begin{remark}
%\red{(How is this the same or different than the simplicial join?  When I first presented this at the AMS Sectional Meeting in ``Wuss-ter" Graham asked a question that I think was this one.)} {\blue (I don't understand this question. Are you talking about how $Y(\gl)$ goes to $Y(\gl)$? It seems to me that $\sigma$ is just a set not a simplicial complex and you are just adding an element to it.)} \red{(I am generally talking about $X(\gs) \times X(\gt) \mto X(\gs\cup\gt^{+m})$.  At some point I thought that there was a notion of simplicial join for arrangements -- now I am not so sure.  The only reason to include such a remark is to not make us look dumb for writing a definition that people know under a different name and not citing that other name.  I'll look into it a bit more and if nothing comes up this remark will not exist.)} {\blue Good call. I don't know of such a thing, but that means absolutely nothing. I VERY briefly searched the internet and didn't see anything. IF we can make anything precise and add  reference that would be great, but I don't know of such a thing.}
%\end{remark}

We can form the direct limit of the following system
\begin{align}
\mc{A}_{2, \ell-k+2} \to \cdots  \to \mc{A}_{k-1,\ell-1} \to \mc{A}_{k,\ell} \to \mc{A}_{k+1,\ell+1} \to \cdots
\end{align}

where each map is multiplication on the right by $\mc{A}_{1,1}$.  Since the difference $\ell-k$ remains fixed we denote the direct limit by $\mc{A}^\infty_{\ell-k}$.  Moreover, the equivalence classes of $\mc{A}^\infty_{\ell-k}$ are represented by the set of $Y(\gL)$.

The following diagram commutes.
$$%\begin{aligned}
\xymatrix @C=0cm{
\mc{A}_{k,\ell} \ar@{=}[dd] & \times  & \mc{A}_{t,m} \ar[dd]^{\times \mc{A}_{1,1}} \ar[rrrrrr] &&&&&& \mc{A}_{k+t,\ell+m} \ar[dd]^{\times \mc{A}_{1,1}} \\ \\
\mc{A}_{k,\ell} & \times & \mc{A}_{t+1,m+1} \ar[rrrrrr]&&&&&& \mc{A}_{k+t+1,\ell+m+1}
}
%\end{aligned}
$$

So there is a left action of $\mc{A}^\infty_{k,\ell}$ on the set of all $\mc{A}_q$ by
\begin{align}
\mc{A}_{k,\ell} \times \mc{A}^\infty_q \to \mc{A}^\infty_{q + \ell - k}.
\end{align}

%%%%%%%%%%%%%%%%%%%%%%%%%%%%%%%%%%%%%%%
%%%%%%%%%%%%%%%%%%%%%%%%%%%%%%%%%%%%%%%

\section{Limits of relative atomic complexes}\label{alg}

All of these arrangement maps induce maps on relative atomic complexes.  One must only replace $\mc{A}$ with $A$, and $X(\gs)=Y(\gl)$ with $a_\gs=b_\gl$.  The equivalence classes of the limiting complex, $\mathfrak{A}_q$, can be represented by the set of all $b_\gL$.

%\begin{align}
%A_{k,\ell} \times A_{t,m} & \to A_{k+t,\ell+m} \\
%a_{\gs} \times a_{\gt} & \to a_{\gs \cup \gt^{+m}}.
%\end{align}
%We have a sequence of algebras
%\begin{align}
%A_{0, \ell-k} \cdots  \to A_{k-1,\ell-1} \to A_{k,\ell} & \to A_{k+1,\ell+1} \to \cdots \\
%a_\gs & \mto a_{\gs \cup \{\ell +1\}}
%\end{align}
%whose direct limit we denote by $A_{k,\ell}$.
%We also have a module action
%\begin{align}
%A_{k,\ell} \times A_r \to A_{r + \ell - k}
%\end{align}

\begin{remark}
If we multiply on left instead of the right, we can get similar results with all of the left and right sides switched.
\end{remark}

Multiplying an element $a_\gs\in A_{k,\ell}$ by $A_{1,1}$ raises the degree, but as with the arrangements, using the complimentary representatives $b_\gl \in A_{k,\ell}$ there is a notion of degree that is constant after some finite stage.
\begin{definition}
Define the codegree of $b_\gL \in A_{k,\ell}$ to be
\begin{align}
\text{codeg}(b_\gL) & = \big[ \text{maximal degree of } A_{k,\ell} \big] - \text{deg}(b_\gL)
\end{align}
\end{definition}

The codegree can be expressed using the set $\gL$ and $k$ and $\ell$.
\begin{proposition}
Let $b_\gl \in A_{k,\ell}$, then
\begin{align} \text{codeg}(b_\gL)= 2 \left| \cap \gL \right| + |\gL| - \left\lceil \frac{\ell-1}{k-1}\right\rceil.
\end{align}
\end{proposition}

\begin{proof}
\begin{align}
\text{codeg}(b_\gL) & = \big[ \text{maximal degree of } A_{k,\ell} \big] - \text{deg}(b_\gL) \\
& = \left[ 2(\ell-1) - \left\lceil \frac{\ell-1}{k-1} \right\rceil \right] - \left[ 2 \text{codim}(X(S)) - |S| \right] \\
& = 2 \left[ \text{dim}(X(S)) - 1 \right] - \left[ \left\lceil \frac{\ell-1}{k-1}\right\rceil - |S| \right]
\end{align}
Now we calculate the dimension of $X(S)$ in terms of $\gl$.
\begin{align}
\text{dim}(X(S)) & = \ell - \left( \left| \bigcup_{\gs \in S} \gs \right| - 1 \right) \\
& = \ell - \left| \bigcup_{\gl \in \gL} \gl^c \right| + 1\\
& = \left| \left(\bigcup_{\gl \in \gL} \gl^c\right)^c \right| +1 \\
& = \left| \bigcap_{\gl\in\gL} \gl \right| = |\cap \gL|
\end{align}
\end{proof}

After the finite stage of the sequence when $\left\lceil \frac{\ell-1}{k-1}\right\rceil = 2$ the codegree of $b_\gL$ is constant.  Therefore we can make the following definition.

\begin{definition}
Let codegree of $b_\gL \in \mathfrak{A}_{q}$ be
\begin{align}
\text{codeg}(b_\gL) = 2 \big( \left| \cap \gL \right| -1 \big) + |\gL| .
\end{align}
\end{definition}

We also define a differential on the limiting complex $\mathfrak{A}_{q}$ that is the limit of the differential in the relative atomic complexes in the directed system.

\begin{definition}
$$d(b_\gL) = \sum (-1)^j b_{\gL \setminus \gl_j}$$
where the sum is over all $\gl_j \in \gL$ with $\cap\gL = \bigcap( \gL \setminus \gl_j )$.
\end{definition}

In order to compute the homology of $\mc{A}_{q}$ with respect to $d$ we need some notation and a lemma.  Let $S$ be a set of atoms in $\mc{A}_{q+k,k}$ for some $k$.  For $\gs \in S$, let $\mc{F}(\gs)$ be the set of $i \in\gs$ that are not contained in any $\gt \in S\bs\gs$.  Let $\mc{P}(S)$ be a subset of $\cup S$ that has the following properties:
\begin{enumerate}[label=(\roman*)]
\item $|\mc{P}(S)|=k-1$
\item $(\cap S) \subset \mc{P}(S)$
\item \label{finger} $\mc{F}(\gs) \bs (\mc{P}(S) \cap \mc{F}(\gs))$ is non-empty for every $\gs \in S$.
\end{enumerate}
There may be more than one set that satisfies the required properties for $\mc{P}(S)$, or there may no such sets.
%These are ``fingers" and ``palms."

The next lemma describes the basic homologies we need for the proof of the next theorem.
\begin{lemma}\label{tohands}
Let $S$ be an independent set of atoms in $A_{q+k,k}$ for which there exists a set $\mc{P}(S)$ satisfying the conditions above.  Assume also that every pair of atoms in $S$ have a non-empty intersection and $\mc{F}(\gs)$ is non-empty for every $\gs \in S$.  Then for
$$\tilde \gs = \mc{F}(\gs) \cup \left( \cap S \right) \cup T$$
with $|\tilde \gs|=k$ and $T \subset \mc{P}(S)$, there is a homology between $a_S$ and $ \pm a_{(S\bs\gs ) \cup \tilde \gs}$.
\end{lemma}

\begin{proof}
To ensure that $\tilde \gs$ exists we note that there is such a $T \subset \mc{P}(S)$ since $\mc{F}(\gt)$ is non-empty and $|\mc{P}(S)| = k-1$.  Since $\mc{P}(S) \subset (\cup S)$, from the construction of $\tilde\gs$ we deduce that
$$\vee (S\cup \tilde \gs) = \vee S = \vee \left((S\bs \gs ) \cup \tilde \gs\right).$$
Moreover, for every $\gt \in S$ with $\gt \neq \gs$ we know $\mc{F}(\gt)$ is non-empty, so
 $$\vee ((S \cup \tilde \gs ) \bs \gt) \neq \vee (S\cup \tilde \gs).$$
Now from the definition of the differential we conclude that 
 $$d(a_{S\cup \tilde \gs} )= a_S \pm a_{(S\bs \gs ) \cup \tilde \gs}$$ 
and hence $a_S$ is homologous to $\pm a_{(S\bs \gs ) \cup \tilde \gs}$.
\end{proof}

\begin{theorem} \label{homology}
The homology of $\mathfrak{A}_{q}$ is generated by the set of equivalence classes of $a_S \in A_{q+k,k}$ such that $q\leq k-1$ and for some choice of $\mc{P}(S)$ 
\begin{align}\label{hand}
\cup S= \mc{P}(S) \cup \left( \bigcup_{\gs \in S} \mc{F}(\gs) \right)
\end{align}
and $|\mc{F}(\gs)| = 1$ for all $\ga \in S$.
\end{theorem}
%The equation in the theorem describes ``hands with fingers" or something close enough to it.  The condition on $\mc{F}(\gs)$ says the ``fingers have exactly one knuckle."
Notice that the requirements the theorem puts on $S$ imply that $\mc{P}(S) = \cap S$ and $|\cap S| = k-1$.

\begin{proof}

By the universal property of a direct limit we may represent every element of the kernel of $d$, and hence the homology, by a representative from a finite stage $A_{q+k,k}$.  In particular, we may take a representative in some $A_{q+k,k}$ for which $q\leq k-1$.  In Theorem 8.8 of \cite{Yuz-Small}, Yuzvinsky proves that the homology of all $A_{\ell,k}$ is generated by independent sets of atoms.  We note that, if $\mc{F}(\gs)$ is non-empty for all $\gs \in S$, the set $S$ is independent.  Thus the sets $S$ described in the theorem are independent sets, and it suffices for us to show that, given an independent set of atoms $S$ in $\mc{A}_{q+k,k}$, the corresponding algebra element is homologous to a $a_{\tilde S}$ for some $\tilde S$ of the type described in the theorem.

Take a representative, $a_S$, for a homology class that is at a finite stage $A_{q+k,k}$ where $q \leq k-1$, and $S$ is an independent set of atoms. The hypergraph associated to $S$ can have at most one component because every two edges intersect when $q \leq k-1$.  

Since $S$ is an independent set of atoms $\vee_{\gt \in S} \gt \neq \vee_{\gt \in S\bs \gs} \gt$ for all $\gs \in S$.  In general, removing $\gs$ from the independent set $S$ may create more connected components in the hypergraph, reduce the number of vertices in $\cup S$, or both.  Because $q\leq k-1$ the hypergraph associated to any set of atoms has only one connected component.  This implies that for $\gs \in S$ there is some element $i \in [q+k]$ with $i \in \gs$ but $i \not\in \gt$ for all $\gt \neq \gs$.  Thus, $\mc{F}(\gs)$ is non-empty for all $\gs \in S$.

We now choose a set $\mc{P}(S)$.  Let $S = \{ \gs_1 , \ldots , \gs_r\}$ and let $i$ be a fixed element of $\mc{F}(\gs_1)$.  The set $\gs_1 \bs i$ satisfies the conditions for $\mc{P}(S)$ and so we set $\mc{P}(S)= \gs_1 \bs i$.  

We now use the homologies described in Lemma \ref{tohands} to show that $a_S$ is homologous to $\pm a_{\tilde S}$ for some $\tilde S$ with the properties described in the statement of the theorem.  First apply Lemma \ref{tohands} to the atom $\gs_1 \in S$ and let $\tilde S_1 = (S\bs \gs_1) \cup \tilde \gs_1= \{ \tilde \gs_1 , \gs_2, \ldots , \gs_r\}$.  By the lemma $a_S$ is homologous to $\pm a_{\tilde S_1}$.  We note that after applying the lemma $\mc{F}(\gs_1) = \mc{F}(\tilde \gs_1)$.  Though we cannot say the same for $\gs_j$ where $j \neq 1$, condition \ref{finger} for $\mc{P}(S)$ ensures that $\mc{F}(\gs_j)$ remains non-empty.

Now we construct a set $\tilde S$ by keeping $\mc{P}(S)$ fixed and inducting on the subscripts of the $\gs$ creating a $\tilde\gs_j$ for each $\gs_j \in S$.   So $\tilde S_j = \{ \tilde \gs_1 , \ldots , \tilde\gs_j , \gs_{j+1} , \ldots , \gs_r \}$ and we can proceed to carry out this construction for $\gs_{j+1}$ as an element of $\tilde S_j$.

After the induction is complete we let $\tilde S = \tilde S_r$.  At each stage of the induction there is a homology between $a_{\tilde S_{j-1}}$ and $\pm a_{\tilde S_j}$, whence $a_S$ is homologous to $\pm a_{\tilde S}$.  Furthermore, when Lemma \ref{tohands} is applied to $\gs_j \in \tilde S_{j-1}$ we have $\mc{F}(\gs_j) = \mc{F}(\tilde \gs_j)$ and condition \ref{finger} ensures that $\mc{F}(\gs_i)$ remains non-empty for all $i \neq j$.  Thus for $\mc{F}(\tilde \gs_j)$ is non-empty for every $\tilde\gs_j$ when considered as a set of $\tilde S$.  The definition of $\tilde \gs_j$ implies that $\gs_j \subset ( \mc{F}(\gs_j) \cup \mc{P}(S))$ so the hypergraph of $\tilde S$ satisfies Equation (\ref{hand}). 
 
We note that a different choice of $\mc{P}(S)$ may yield a different set $\tilde S$ but every choice yields a $\tilde S$ with the desired properties.  In particular, all possible choices lead to an $a_{\tilde S}$ that satisfies the conditions of the theorem and is homologous to the original $a_S$. 
 
We may now assume that the homology representative $a_S$ satisfies Equation (\ref{hand}) and $\mc{F}(\tilde \gs)$ is non-empty for all $\tilde \gs \in \tilde S$.  To finish the proof we show that, in order for the homology class of $a_S$ to be non-trivial, $|\mc{F}(\gs)| =1$ for all $\gs \in S$.  Assume that there is some $\gs \in S$ with $|\mc{F}(\gs)|>1$.  Then we can consider $S \cup \gm$ where $\gm$ contains the $k-1$ vertices of a set $\mc{P}(S)$ and one of $\mc{F}(\gs)$.  Now $da_{S \cup \gm} = a_S$.  Therefore if a $|\mc{F}(\gs)| >1$ the element $a_S$ represents the zero homology class.  
\end{proof}

\begin{remark}
It would be interesting to know what topological maps between complements of arrangements induce these maps on relative atomic complexes, if such topological maps exist.
\end{remark}

%%%%%%%%%%%%%%%%%%%%%%%%%%%%%%%%%%%%%%%
%%%%%%%%%%%%%%%%%%%%%%%%%%%%%%%%%%%%%%%

\section{A bi-complex}\label{Bi}

Let $\mathfrak{A}_*$ denote the graded vector space $\bigoplus\limits_{q\geq0} \mathfrak{A}_q$.  Note that $\mathfrak{A}_0$ is generated by $b_\emptyset$.  An index set $\gL$ for elements $b_\gL \in \mathfrak{A}_*$ is uniquely determined by the pair $(\cap\gL, \hat\gL)$ where $\hat \gL = \{ \gl \setminus ( \cap\gL) \, | \, \gl \in \gL\}$.  Every element $\hat \gl \in \hat\gL$ has the same cardinality, $|\hat \gl| = |\gl| - | \cap \gL|$. We identify the set of all possible $\cap\gL$ as the power set of $\Z_{>0}$ to define a simplicial differential on $\mathfrak{A}_*$.

\begin{definition}
Let $\gd \co \mathfrak{A}_* \to \mathfrak{A}_*$ be defined by
\begin{align}
\gd \co \mathfrak{A}_q  & \to \mathfrak{A}_{q-1} \\
b_{(\cap \gL, \hat\gL)} & \mto \sum_{\ell_i \in \cap_j \gl_j} (-1)^i b_{(\cap\gL \bs \ell_i, \hat \gL)}.
\end{align} 
\end{definition}

It follows immediately from the definition that $b_{(\cap \gL , \hat \gL)}$ maps to other elements of the form $b_{(\mc{I} , \hat \gL)}$ under $\gd$. 

\begin{remark}For the remainder of this paper we will occasionally abuse notation by writing $b_{\gL\bs \ga}$ to mean $b_{(\cap \gL \bs \ga,\hat \gL)}$ for an element $\ga\in \cap\gL$.  Note that this is a very different operation than the one in the differential $d$ where we take out an entire set of $\gL$. 
\end{remark}

\begin{proposition}
The map $\gd$ is a differential on the set $\mathfrak{A}_*$.
\end{proposition}

The proof is a standard check of signs.  We note that this differential is defined in the same manner as standard simplicial differential.

Let $ \bar{\mathfrak{A}}_*$ be the sub-algebra of $\mathfrak{A}_*$ generated by all elements $b_\gL$ where $|\gL|=1$.  So every element of $\bar{\mathfrak{A}}_*$ is a linear combination of elements of the form $b_\gL=b_{\{\gl\}}$ which we denote by $b_\gl$.

Let $\gD_+$ be the standard infinite-dimensional simplex with the empty set as a $-1$ simplex.  Since the realization of $\gD_+$ is contractible it's homology is zero and the augmentation by a $-1$ simplex means it is zero in degree zero and degree -1 as well.
%The realization is contractible by the homotopy $(1-t)e_1 + t x$ for any $x \in |\gD_+|$.)
Given a simplex, $\gh \in \gD_+$ we identify it as a subset of $\Z_{>0}$ and choose a generator in $C_*(\gD_+)$ associated to $\gh$ and denote it by $\bar\gh$. 

\begin{lemma} \label{barA}
There is an isomorphism of DGAs between $(\bar {\mathfrak{A}}_*, \gd)$ and $C_*(\gD_+)$.
\end{lemma}

\begin{proof}
The isomorphism $\ga \co\bar {\mathfrak{A}}_* \to C(\gD_+)$ sends $b_\gl$ to $\bar\gl$ so $b_\emptyset$ maps to the algebra element associated to the $-1$ simplex.  The fact that the map is a bijection is obvious.  The fact that it is a map of chain complexes is also obvious because both differentials are defined in the standard simplicial way.
\end{proof}

 Let $\mathscr{P}_i(\Z_{>0})$ denote the subsets of $\Z_{>0}$ with cardinality $i$.  Further, let $\mathscr{E}_i$ denote the subsets of $\mc{L} \subseteq \mathscr{P}_i(\Z_{>0})$ for which $\cap\mc{L}$ is empty.  Now, for $\mc{L} \in \mathscr{E}_i$, let $\mathfrak{A}_*(\mc{L})$ be the subalgebra of $\mathfrak{A}_*$ generated by elements of the form $b_{(\mc{I} , \mc{L})}$ where $\mc{I} \subset \Z_{>0}$ and $\mc{I} \cap (\cup \mc{L})$ is empty. Note that when $i=0$ then the only possible $\mc{L}$ is the empty set. And so in this notation we have that $\mathfrak{A}_*(\emptyset )=\bar{\mathfrak{A}}_*$.

\begin{lemma} \label{bars}
For a fixed $\mc{L}\subseteq \mathscr{E}_i$ the unique order preserving bijection $$\gF_{\mc{L}} \co\Z_{>0} \bs (\cup\mc{L}) \to \Z_{>0}$$ induces an isomorphism $\bar{\gF}_{\mc{L}} :\mathfrak{A}_{*}(\mc{L}) \to \bar{\mathfrak{A}}_{* -i}$ given by 
$$\bar \gF_{\mc{L}} \left(b_{(\mc{I}, \mc{L})}\right) =  b_{\gF_{\mc{L}}(\mc{I})}.$$ 
\end{lemma}

\begin{proof}
The generators of $\bar{ \mathfrak{A}}_*$ are indexed by subsets of $\Z_{>0}$.  For a fixed $\mc{L}$ we consider the set of all $\gL$ for which $\hat \gL = \mc{L}$.  Generators of $\mathfrak{A}_*(\mc{L})$ of the form $b_{(\cap \gL , \mc{L})}$ are indexed by $\cap \gL$ and hence the subsets $\gF_{\mc{L}}(\cap\gL)$ of $\Z_{>0}$.  Since $\gF_{\gL}$ is a bijection, the generators of the two algebras are in bijection with each other.  The differentials are the same, so we have an isomorphism of algebras.
\end{proof}

\begin{theorem}\label{decomp}
The complex $(\mathfrak{A}_*, \gd)$ is isomorphic to a direct sum of simplicial chain complexes
$$ (\mathfrak{A}_*, \gd)  \cong \left( \bigoplus_{i\geq 0} \bigoplus_{\mc{L} \subset \mathscr{E}_i} C_*(\gD_+) \right).$$
\end{theorem}

\begin{proof}

We denote by $\bar\gb(i,\mc{L})$ the element $\bar\gb \in C_*(\gD_+)$, where $C_*(\gD_+)$ is the summand corresponding to a fixed $i$ in the first summation and a fixed $\mc{L}$ in the second summation of $\left( \bigoplus_{i\geq 0} \bigoplus_{\mc{L} \subset \mathscr{P}_i(\Z_{>0})} C_*(\gD_+) \right)$.

Consider some $\gL = (\cap\gL, \hat\gL)$.  If we let $\gd b_\gL = \sum c_j b_{\gL_j}$ then $\hat \gL = \hat \gL_j$ for all $j$. We prove the theorem by exhibiting an isomorphism $\gf$ and its inverse $\gy$. The map 
$$\gf \co (\mathfrak{A}_*, \gd)  \lto \left( \bigoplus_{i\geq 0} \bigoplus_{\mc{L} \subset \mathscr{E}_i} C_*(\gD_+) \right)$$ 
 is defined on a generator $b_{(\cap\gL, \hat\gL)}  \in (\mathfrak{A}_*,\gd)$ by 
\begin{equation}
\gf\left(b_{(\cap\gL, \hat\gL)}\right) = \bigg[\left(  \ga\circ\bar\gF_{\hat\gL} \right) \left( b_{(\cap\gL, \hat\gL)} \right)\bigg]\left( |\hat\gl| , \hat\gL \right) .
\end{equation} where the map $\ga$ is defined in the proof of Lemma \ref{barA} and the map $\bar\gF_{\hat\gL}$ is defined in Lemma \ref{bars}.  We note that ${b}_{\gF_{\hat\gL}(\cap\gL)}$ is an element of $\bar{\mathfrak{A}}_*$ and the element 
$$\left(  \ga\circ\bar\gF_{\hat\gL} \right) \left( b_{(\cap\gL, \hat\gL)} \right) = \ga\left( b_{\gF_{\hat\gL}(\cap\gL)} \right)= \overline{\gF_{\hat\gL}(\cap\gL)}, $$
is in $C_*(\gD_+)$.  The second part of the notation $( |\hat\gl| , \hat\gL )$ gives the indices for the appropriate $C_*(\gD_+)$ summand in the codomain.  We extend this definition by linearity to all elements of $\mathfrak{A}_*$
 
Now we define $\gy$, the inverse to $\gf$ on a generator $\bar\gb(i,\mc{L})$
%$$\bar\gb(i,\mc{L}) \in \left( \bigoplus_{i\geq 0} \bigoplus_{\mc{L} \subset \mathscr{P}_i(\Z_{>0})} C_*(\gD_+) \right)$$ 
by the equation
\begin{equation}
\gy( \bar\gb(i,\mc{L})) = \bigg( \bar\gF_{\mc{L}}^{-1} \circ \ga^{-1} \bigg)  \left( \bar\gb \right) = \bar\gF_{\mc{L}}^{-1}  \left(b_{\gb} \right),
\end{equation} 
where $b_\gb$ is in $\bar{\mathfrak{A}}_*$. Again we extend this definition by linearity to all elements in the domain.

For $b_{(\cap\gL,\hat{\gL})}$, the differential $\gd$ only changes the intersection $\cap \gL$ while every term has $\hat{\gL}$ remain the same . So, for a fixed $\mc{L}$ the differential $\gd$ preserves the complement of the intersection $\mc{L}$ and hence it is the same differential on the subalgebras $\gd: \mathfrak{A}_q(\mc{L})\to\mathfrak{A}_{q-1}(\mc{L})$. The construction of the subalgebras $\mathfrak{A}_*(\mc{L})$ together with this fact about the differential give that the complex decomposes into the direct sum:

$$(\mathfrak{A}_*,\gd)\cong \bigoplus_{i\geq 0}\bigoplus_{\mc{L}\in \mathscr{E}_i} (\mathfrak{A}_*(\mc{L}),\gd).$$

Now both $\gf$ and $\gy$ are DGA maps since they map summands to summands. Since Lemma \ref{barA} and Lemma \ref{bars} show that for each summand $\gf$ and $\gy$ are isomorphisms and inverses of each other, we have completed the proof.
\end{proof}

\begin{corollary} \label{zero}
The homology of $\mathfrak{A}_*$ with respect to $\gd$ is identically zero.
\end{corollary}

This follows directly from the fact that the homology $H_*(\gD_+)$ is zero and Theorem \ref{decomp}. We now combine the simplicial structure with the structure inherited from the relative atomic complexes.
\begin{proposition}
$\{\mathfrak{A}_* , d , \gd \}$ is a double complex.
\end{proposition}

To calculate the total homology of the bi-complex we define a new grading on the collection $\mathfrak{A}_*$. 

\begin{definition}
 Let $\mathfrak{A}^{(m,t)}$ be generated by the set
\begin{align} 
\left\{ b_\gL \in \mathfrak{A}_* \st |\gL|=m \, , \; |\cap\gL|=t \right\}.
\end{align}
\end{definition}

\begin{proposition}
The total homology $H^{\text{Tot}}(\mathfrak{A}_*,d, \gd)$ is zero.
\end{proposition}

\begin{proof}
Notice that 
\begin{align} \label{r1delta}
\gd \co \mathfrak{A}^{(m,t)} \to \mathfrak{A}^{(m,t-1)}
\end{align} 
and by Corollary \ref{zero} the homology, with respect to $\gd$, is zero.  Thus the spectral sequence of the bicomplex filtered by $t$-degree (columns) has an $E_1$ page that is identically zero.  
\end{proof}

%
%
%%%%%%%%%%%
% The stuff below may be true and necessary if we don't include the augmentation as described the remark above.
%%%%%%%%%%%
%
%
%Each homology generator $[b_\gL]_\gd$ has a representative $b_\gL$ with $\cap\gL = \{j\}$ for every $j\geq 1$.
%
%The differential $d$ does not change $\cap \gL$ and can be non-zero since 
%\begin{align}\label{t1d}
%d \co A(m,r) \to A(m-1,t).
%\end{align}  
%It suffices to consider $d_1 = [ d]_\gd$ restricted to the elements $[b_\gL]_\gd$ with $\cap\gL=\{1\}$.  But there is only one representative of $[b_\gL]_\gd$ with $\cap\gL=\{1\}$ and it is determined by $\hat\gL$, since $\gL$ is determined by the pair $(\cap \gL, \hat \gL)$ and every copy of $C_*(\gD)$ maps to a single $\hat\gL$.
%
%This means that $d_1=[d]_\gd$ is just the restriction of $d$ to the subset of $b_\gL$ with $\cap\gL=\{1\}$.
%
%Theorem \ref{homology} computes the homology of $A^(m,r)$ with respect to $d$.  Whence restricting to $\cap\gL=\{1\}$ we have that the $E_2$ page of the spectral sequence contains only ``hands with fingers with one knuckle" with $\cap\gL=\{1\}$.
%
%Finally since all surviving classes are on the line where $r=1$ and (\ref{t1delta}) there can be no further differentials and $E_2=E_\infty$.
%
%\red{This may be true of the grading below in the not-augmented case: The homology $H^*(A(m,n);d)$ is non-zero exactly when $m=n$ and $H^{\text{Tot}}(A(m,n);d,\gd)$ is non-zero exactly when $m=n$ and $|\cap\gL|=1$ (these follow from the theorems above and doing some degree counting).}
%
%

There is at least one other interesting bi-grading on $\mathfrak{A}_*$.  This bi-grading behaves nicely with respect to both differentials and helps to organize the homology of $\mathfrak{A}_*$ with respect to $d$ as described in Theorem \ref{homology}.
\begin{definition}
We define a bi-grading on the direct sum of algebras $\bigoplus_{2 \leq k \leq \ell} A_{\ell,k}$
\begin{align}
A_\ell(m,n) & = \left\{ \gb \in A_{\ell_0,k} \;\bigg|\; \ell_0 \leq \ell, \; \gb = \sum_\gL b_\gL \text{ and }  m=\ell-k+1 \, , \; n=|\gL|+|\cap\gL| \right\}.
\end{align}
This grading extends to the limiting algebra $\mathfrak{A}_*$ by letting $\ell$ go to infinity.
In the finite algebras we set $q=\ell-k$, as usual.
\end{definition}

Let $\mathfrak{A}(m,n)$ be the direct limit of $A_\ell(m,n)$ under $\times A_{1,1}$.
With these gradings the differentials behave as follows:
\begin{align}
d\co & A_\ell(m,n) \to A_\ell(m,n-1) \\
\gd \co & A_\ell(m,n) \to A_\ell(m-1,n-1) \\
A_{1,1}\co & A_\ell (m,n) \to A_{\ell+1}(m,n)
\end{align}
The multiplication by $A_{1,1}$ map sends independent sets to independent sets.  Moreover, it maps a set of the form described in Theorem \ref{homology} to a set of the same form.  The index $m$ of $\mathfrak{A}(m,n)$ can be interpreted as the maximum size of the set of atoms $S$, and $n$ can be understood as the size of $\mc{P}(S)$.  The homology is non-zero only when $m \geq n$.  The same is true for $A_{\ell,k}$ when $\ell \geq 2k-1$.

%%%%%%%%%%%%%%%%%%%%%%%%%%%%%%%%%%%%%%%
%%%%%%%%%%%%%%%%%%%%%%%%%%%%%%%%%%%%%%%

\section{A graded product}\label{gproduct}

A cursory check shows that a cup product structure on $\mathfrak{A}_q$ that is the limit of the cup product on $A_{k,\ell}$, is always equal to zero.  There is, however, a different product structure on $\bigoplus_{q \in \Z_{>0}}\mathfrak{A}_q$ that nearly commutes with the differential $\gd$. In this section we define this product on the collection of all the $\mathfrak{A}_q$ and study some of its properties.  In order to define a product on $\mathfrak{A}_*$ we will need the following combinatorial definition.

\begin{definition}

For $\gL=\{\gl_1,\dots ,\gl_r\}$ and $\gG=\{\gg_1,\dots, \gg_s\}$ with $\gl_i$ and $\gg_j$ subsets of $\Z_{>0}$ we say that $\gL$ and $\gG$ are \emph{compatible} if either one of sets $\gL$ or $\gG$ is empty or both $r=s$ and $$\left[ \cup\gL\right] \cap \left[ \cup\gG\right]=\emptyset.$$ If $\gL$ and $\gG$ are compatible we define $$\gL \biguplus \gG=\{\gl_1 \cup \gg_1,\dots,\gl_r\cup\gg_r\}.$$

\end{definition}

Next we will need a little notation to define the sign of the product. Recall that for any two elements $\gl,\gg \in [\ell]$, the sign $\ge (\gl,\gg)$ is the sign of the permutation that reorders $\gl \cup \gg$ so that all the elements of $\gl$ come before those in $\gg$ considered as the number of inversions not just the parity. Then the sign of two compatible multisets $\gL=\{\gl_1,\dots ,\gl_r\}$ and $\gG=\{\gg_1,\dots, \gg_r\}$ is $$\epsilon (\gL,\gG)=\ge(\cap \gL,\cap \gG).$$

Now we can define a product on the vector space $A$.

\begin{definition}\label{prod1} 
For simple elements $b_\gL\in \mathfrak{A}_p$ and $b_{\gG}\in \mathfrak{A}_q$ let 
\begin{equation*}
b_\gL b_\gG:=\begin{cases} (-1)^{\epsilon (\gL,\gG)}b_{\gL\biguplus \gG} \quad& \text{if } \gL \text{ and }\gG \text{ are compatible}\\
0 &\text{otherwise.}
\end{cases} \end{equation*} 
Then extend this product linearly to all of $\mathfrak{A}_*$.
\end{definition}

In Definition \ref{prod1}, since the size of each element of $\gL$ and $\gG$ is $p$ and $q$ respectively the product $b_\gL b_\gG\in \mathfrak{A}_{p+q}$ is graded (but not commutative nor graded commutative). However, it is associative.

Let $b_\gL\in \mathfrak{A}_p$, $b_{\gG}\in \mathfrak{A}_q$, and $b_\Theta\in \mathfrak{A}_t$. If any pair of the sets $\gL$, $\gG$, and $\Theta$ are not compatible then both products $b_\gL (b_\gG b_\Theta)$ and $(b_\gL b_\gG) b_\Theta$ will clearly be zero. Hence, we assume that any pair of the sets $\gL$, $\gG$, and $\Theta$ are compatible. Since any pair is compatible, any of these sets will be compatible with the union of the other two and both products $b_\gL (b_\gG b_\Theta)$, and $(b_\gL b_\gG) b_\Theta$, will result in $\pm b_{\gL \cup \gG \cup \Theta}$.  The sign is associative, so the product of Definition \ref{prod1} is associative.

One reason to investigate this product is that it is close to satisfying the graded Leibniz rule with respect to $\gd$. To show this we setup some notation and prove a few Lemmas. Suppose that $\gL$ and $\gG$ are compatible, $|\gl_i|=p$, $|\gg_i|=q$, $\cap \gL =\{\alpha_1,\dots ,\ga_x\}$, and $\cap \gG =\{\gb_1,\dots ,\gb_y\}$. Since $\gL$ and $\gG$ are compatible $\cap [\gL \biguplus \gG]=[\cap \gL]\cup[\cap\gG]$. 

\begin{lemma}\label{inversions} 
For $i\in \{1,\dots ,x\}$ let $k_i$ be the number of elements strictly between $\alpha_i$ and $\alpha_{i+1}$ in the intersection of the union $\cap [\gL \biguplus \gG]$. Then remembering the sign as the number of inversions for reordering the sets yields
$$\epsilon (\gL \bs \{ \alpha_{i+1}\} ,\gG)=\epsilon (\gL \bs \{\alpha_i\},\gG)-k_i.$$
\end{lemma}

\begin{proof}
If an element is deleted from a reordering, then the only change to the sign of the permutation is that all the transpositions with that element are deleted. So for any $\alpha_i$ $$\epsilon (\gL \bs \{\alpha_i\},\gG)=\epsilon (\gL,\gG)-\epsilon (\{\alpha_i\},\gG).$$ Note that $\epsilon (\{\alpha_i\},\gG)$ is the number of elements in $\gG$ that come before $\alpha_i$. Then 
\begin{align*}
\epsilon(\gL\bs\{\alpha_{i+1}\},\gG) & =\epsilon(\gL,\gG)-\epsilon(\{\alpha_{i+1}\},\gG) \\
& =\epsilon (\gL,\gG)-[\epsilon(\{\alpha_i\},\gG)+k_i] \\
& =\epsilon(\gL\bs \{\alpha_i\},\gG)-k_i.
\end{align*}
\end{proof}

Let the intersection of the union, $\cap [\gL \biguplus \gG]$, be the set $\{\xi_1,\dots ,\xi_s\}$ where the $\xi_i$ are ordered. Define a function $\gs \co [\cap \gL]\cup[\cap\gG] \to [s]$ by $\gs(\alpha_i)=j$ when $\alpha_i=\xi_j$ and $\gs(\gb_s)=t$ when $\gb_s=\xi_t$.

\begin{lemma}\label{alternating}

As $i\in \{1,\dots, x\}$ increases the sum $\gs (\alpha_i)+\epsilon(\gL\bs\{\alpha_i\},\gG)$ alternates its parity.

\end{lemma}

\begin{proof}

There are many cases for this Lemma. We will prove just one and note that all the others are just permutations of the words ``even" and ``odd".  Suppose that $\gs (\alpha_i)$, $\epsilon(\gL\bs \{\alpha_i\},\gG)$, and $\gs(\alpha_{i+1})$ are even. We want to show that $\epsilon(\gL\bs \{\alpha_{i+1}\},\gG)$ is odd. Since $\gs(\alpha_i)$ and $\gs (\ga_{i+1})$ are both even the number of elements between them in the union $\gL \biguplus \gG$, which is $k_i$ in the notation of Lemma \ref{inversions}, is odd. Then using Lemma \ref{inversions} $\epsilon(\gL\bs\{\alpha_{i+1}\},\gG)=\epsilon(\gL\bs \{\alpha_i\},\gG)-k_i$ is an even number subtract an even results in an even number. The other seven cases are very similar.\end{proof}

\begin{remark}\label{symmitricity}Lemmas \ref{inversions} and \ref{alternating} are symmetric over the inputs of the sign function $\epsilon( - , - )$. For example the corresponding statement of Lemma \ref{alternating} is that the parity of the sum $\gs (\beta_i)+\epsilon(\gL,\gG\bs \{\beta_i\})$ alternates.\end{remark}

\begin{theorem}\label{Leibniz}

The differential satisfies the Leibniz rule $$\delta(b_\gL b_\gG)=\gd(b_\gL)b_\gG +(-1)^{|\cap \gL |} b_\gL \gd (b_\gG )$$ whenever the product $b_\gL b_\gG$ is non-zero or $|\gL |\neq |\gG |$ or $[\cup\gL]\cap[\cup\gG]\subseteq [\cap \gL]\cap [\cap \gG]$.

\end{theorem}

\begin{proof} We show that it commutes with the differential for simple elements $b_\gL$ and $b_\gG$ and then extend linearly. Assume that the product $b_{\gL}b_{\gG}$ is non-zero so that the sets $\gL$ and $\gG$ are compatible. Then

\begin{align*} 
\gd (b_\gL b_\gG) & = \gd \left((-1)^{\epsilon(\gL,\gG)}b_{\gL \biguplus \gG} \right)\\
& = (-1)^{\epsilon (\gL,\gG)}\sum\limits_{\xi_i\in \cap [\gL \biguplus \gG]}(-1)^ib_{[\gL\biguplus\gG]\bs \xi_i}\\
& = (-1)^{\epsilon (\gL,\gG)}\left[ \sum\limits_{i=1}^x(-1)^{\gs(\ga_i)}b_{ [\gL\biguplus\gG]\bs \ga_i} +\sum\limits_{j=1}^y (-1)^{\gs(\gb_j)}b_{[\gL\biguplus\gG]\bs \gb_j }\right] \\
& =  (-1)^{\epsilon (\gL,\gG)}\left[\sum\limits_{i=1}^x(-1)^{\gs(\ga_i)+\epsilon (\gL\bs \ga_i,\gG)}(b_{ \gL\bs \ga_i} b_\gG )\right.   \\
& \hspace{1.5in} \left. +\sum\limits_{j=1}^y (-1)^{\gs(\gb_j)+\epsilon (\gL,\gG\bs\gb_j)}(b_\gL b_{\gG\bs \gb_j })\right] 
\end{align*}

Lemma \ref{alternating} says exactly that the sums above alternate, but it does not say at which sign they start. So if we factor out the first sign of each sum plus 1 then they will both start at negative 1 which is the definition of the differential. So, the expression becomes 

\begin{multline}\label{factored} 
\begin{aligned}
\gd (b_\gL b_\gG)= (-1)^{\epsilon (\gL,\gG)}\left[(-1)^{\gs (\alpha_1)+\epsilon (\gL\bs\{\alpha_1\},\gG)+1}\sum\limits_{i=1}^x(-1)^{i}(b_{ \gL\bs \ga_i} b_\gG )\right.\\
\hspace{1in} \left. +(-1)^{\gs (\beta_1)+\epsilon (\gL,\gG\bs \{\beta_1\})+1}\sum\limits_{j=1}^y (-1)^{j}(b_\gL b_{\gG\bs \gb_j })\right]
\end{aligned}
\end{multline}

Notice that $\gs (\alpha_1)-1=\epsilon (\{\alpha_1\},\gG)$. Thus $\gs(\alpha_1)+\epsilon (\gL\bs \{\alpha_1\},\gG)=\epsilon (\gL,\gG)+1$. The second summation term of \ref{factored} is a little more complicated. However we can still decompose via Lemma \ref{inversions}: 

\begin{equation}\label{decomposebeta}
\gs(\beta_1)+\epsilon (\gL,\gG\bs \{\beta_1\})+1= \gs (\beta_1)+\epsilon (\gL,\gG)-\epsilon (\gL,\{\beta_1\})+1.
\end{equation} 

Then note that 
$$\gs (\beta_1)=x-\{ \text{the number of elements from $\gL$  greater than $\beta_1$}\} +1,$$ 
where $x$ is the size of the intersection of $\gL$: $x=|\cap \gL|$. Since 
$$\{ \textrm{the number of elements from $\gL$ greater than } \beta_1 \}=\epsilon(\gL,\{\beta_1\})$$
the expression in Equation \ref{decomposebeta}, up to parity, is 

\begin{align*}
\gs(\beta_1)+\epsilon (\gL,\gG\bs \{\beta_1\})+1& = x-\epsilon (\gL,\{\beta_1\})+1+\epsilon(\gL,\gG)-\epsilon(\gL,\{\beta_1\})+1\\
&= x-2\epsilon(\gL,\{\beta_1\})+\epsilon(\gL,\gG)+2\\
&= x+\epsilon(\gL,\gG)
\end{align*} 

Now putting these signs into Equation \ref{factored} we get 

\begin{align}
\begin{aligned}
\gd (b_\gL b_\gG)&= (-1)^{\epsilon (\gL,\gG)}\left[(-1)^{\epsilon (\gL,\gG)+2}\sum\limits_{i=1}^x(-1)^{i}(b_{ \gL\bs \ga_i} b_\gG )\right.\\
& \hspace{1in} \left. +(-1)^{x+\epsilon (\gL,\gG)}\sum\limits_{j=1}^y (-1)^{j}(b_\gL b_{\gG\bs \gb_j })\right]\\
&=\sum\limits_{i=1}^x(-1)^{i}(b_{ \gL\bs \ga_i} b_\gG )+(-1)^x\sum\limits_{j=1}^y (-1)^{j}(b_\gL b_{\gG\bs \gb_j })\\
&=\left(\sum\limits_{i=1}^x(-1)^{i}b_{ \gL\bs \ga_i}\right) b_\gG +(-1)^xb_\gL\left(\sum\limits_{j=1}^y (-1)^{j}b_{\gG\bs \gb_j }\right)\\
&= \gd(b_\gL)b_\gG +(-1)^x b_\gL \gd (b_\gG ). \\
\end{aligned}
\end{align}

Now suppose that the product $b_\gL b_\gG$ is zero. If $|\gL |\neq |\gG |$ then this property will persist in all the products in the expansion of the Leibniz formula. Hence in those cases both sides of the formula are zero. 

Next assume that  $[\cup\gL]\cap[\cup\gG]\subseteq [\cap \gL]\cap [\cap \gG]$. The only case we have to worry about is when $\left[ \cap \gL\right] \cap \left[ \cap \gG\right]=\{p\}$ is just one element. If this intersection was larger then the products in the expansion of the right side of the Leibniz rule would trivially be zero. We need to show that in the expansion of the Leibniz formula for the differential $\gd$ that the term on the left cancels with the term on the right since all the other products will be zero. Suppose that $\alpha_i=\beta_j=p$ are the locations of the elements in the intersection. The sign of the non-zero term on the left side is $(-1)^{i+\epsilon (\gL\bs \{\alpha_i\},\gG)}$ where the sign on the right is $(-1)^{x+j+\epsilon (\gL,\gG\bs\{\beta_j\})}$. Since there are $x-i$ many inversions in $\gL\bs \{\alpha_i\}$ with $\beta_j$ we have that $\epsilon (\gL\bs\{\alpha_i\},\gG)=\epsilon (\gL\bs\{\alpha_i\},\gG\bs\{\beta_j\})+x-i.$ Similarly since there are $j-1$ elements of $\gG\bs\{\beta_j\}$ less than $\beta_j$ we have that $\epsilon (\gL,\gG\bs\{\beta_j\})=\epsilon (\gL\bs\{\alpha_i\},\gG\bs\{\beta_j\})-(j-1).$ Now plugging this into the formulas for the signs we have the left hand sign is 
$$(-1)^{\epsilon (\gL\bs\{\alpha_i\},\gG\bs\{\beta_j\})+x}$$
and the right hand sign is 
$$(-1)^{\epsilon (\gL\bs\{\alpha_i\},\gG\bs\{\beta_j\})+x+1}.$$ 
Since the two products will be exactly the same and the signs are off by a negative we have that the expansion of the Leibniz formula is also zero. These are all the cases where the Leibniz formula holds. It turns out that it does not hold for all zero divisors.\end{proof}

Let $\gL=\{\{1,2\},\{1,3\}\}$ and $\gG=\{\{2,4\},\{2,5\}\}$. Then certainly $b_\gL b_\gG=0$. But in the Leibniz formula 

\begin{align*}
\gd (b_\gL)b_\gG+(-1)^2b_\gL\gd(b_\gG) & =[-b_{\{\{2\},\{3\}\}}]b_\gG+b_\gL[-b_{\{\{4\},\{5\}\}}] \\
& =b_\gL[-b_{\{\{4\},\{5\}\}}] \\
& =b_{\{\{1,2,4\},\{1,3,5\}\}}\neq 0.
\end{align*} 
Hence the Leibniz formula does not hold in this case.

\begin{remark} We have shown that $(\mathfrak{A}_*,\gd)$ is very close to being a dga, but this last example shows that it is not. Since the homology with respect to $\gd$ is zero, from the perspective of homology, the failings of the Leibniz rule are not of great importance.
\end{remark}

\begin{remark}

This product also does not satisfy the Leibniz formula for the differential $d$. Let $\Lambda = \{ \{ 1,2 \} , \{ 2,3\} , \{ 3,4 \} \}$ and $\Gamma = \{ \{ 6 \} , \{ 7 \} \}$. Then $b_\gL b_\gG=0$ since the number of subsets are not the same. So $d(b_\gL b_\gG)=0$. However $$\pm d(b_\gL)b_\gG\pm b_\gL d(b_\gG)=\pm b_{\{\{1,2\},\{3,4\}\}}b_\gG\pm 0=\pm b_{\{\{1,2,6\},\{3,4,7\}\}}\neq 0.$$

\end{remark}

%%%%%%%%%%%%%%%%%%%%%%%%%%%%%%%%%%%%%%%
%%%%%%%%%%%%%%%%%%%%%%%%%%%%%%%%%%%%%%%

\bibliographystyle{amsplain}
\bibliography{limits}

\providecommand{\bysame}{\leavevmode\hbox to3em{\hrulefill}\thinspace}
\providecommand{\MR}{\relax\ifhmode\unskip\space\fi MR }
\providecommand{\MRhref}[2]{%
  \href{http://www.ams.org/mathscinet-getitem?mr=#1}{#2}
}
\providecommand{\href}[2]{#2}
\begin{thebibliography}{10}

\bibitem{Baskakov-03}
I.~V. Baskakov, \emph{Triple {M}assey products in the cohomology of
  moment-angle complexes}, Uspekhi Mat. Nauk \textbf{58} (2003), no.~5(353),
  199--200. \MR{MR2035723 (2004j:55013)}

\bibitem{BCH}
A.~Jon Berrick, Frederick~R. Cohen, Elizabeth Hanbury, Yan-Loi Wong, and Jie Wu
  (eds.), \emph{Braids}, Lecture Notes Series. Institute for Mathematical
  Sciences. National University of Singapore, vol.~19, World Scientific
  Publishing Co. Pte. Ltd., Hackensack, NJ, 2010, Introductory lectures on
  braids, configurations and their applications, Papers from the International
  Conference held at the National University of Singapore, Singapore, June
  25--29, 2007. \MR{2640539 (2011a:55002)}

\bibitem{BLY-k-equal}
Anders Bj{\"o}rner, L{\'a}szl{\'o} Lov{\'a}sz, and Andrew Yao, \emph{Linear
  decision trees: Volume estimates and topological bounds}, Proc. 24th ACM
  Symp. on Theory of Computing, ACM Press, New York, 1992, pp.~170--177.

\bibitem{BW-96}
Anders Bj{{\"o}}rner and Michelle~L. Wachs, \emph{Shellable nonpure complexes
  and posets. {I}}, Trans. Amer. Math. Soc. \textbf{348} (1996), no.~4,
  1299--1327. \MR{1333388 (96i:06008)}

\bibitem{BW-95}
Anders Bj{\"o}rner and Volkmar Welker, \emph{The homology of ``{$k$}-equal''
  manifolds and related partition lattices}, Adv. Math. \textbf{110} (1995),
  no.~2, 277--313. \MR{MR1317619 (95m:52029)}

\bibitem{B-73}
Egbert Brieskorn, \emph{Sur les groupes de tresses [d'apr{\`e}s {V}. {I}.
  {A}rnol'd]}, S{\'e}minaire {B}ourbaki, 24{\`e}me ann{\'e}e (1971/1972),
  {E}xp. {N}o. 401, Springer, Berlin, 1973, pp.~21--44. Lecture Notes in Math.,
  Vol. 317. \MR{0422674 (54 \#10660)}

\bibitem{DCP95}
C.~De~Concini and C.~Procesi, \emph{Wonderful models of subspace arrangements},
  Selecta Math. (N.S.) \textbf{1} (1995), no.~3, 459--494. \MR{MR1366622
  (97k:14013)}

\bibitem{dLS}
Mark de~Longueville and Carsten~A. Schultz, \emph{The cohomology rings of
  complements of subspace arrangements}, Math. Ann. \textbf{319} (2001), no.~4,
  625--646. \MR{1825401 (2002d:52009)}

\bibitem{DS-Massey}
Graham Denham and Alexander~I. Suciu, \emph{Moment-angle complexes, monomial
  ideals and {M}assey products}, Pure Appl. Math. Q. \textbf{3} (2007), no.~1,
  25--60. \MR{MR2330154}

\bibitem{FY-Formal}
E.~M. Feichtner and S.~Yuzvinsky, \emph{Formality of the complements of
  subspace arrangements with geometric lattices}, Zap. Nauchn. Sem.
  S.-Peterburg. Otdel. Mat. Inst. Steklov. (POMI) \textbf{326} (2005),
  no.~Teor. Predst. Din. Sist. Komb. i Algoritm. Metody. 13, 235--247, 284.
  \MR{MR2183223 (2007a:55021)}

\bibitem{GM88}
Mark Goresky and Robert MacPherson, \emph{Stratified {M}orse theory},
  Ergebnisse der Mathematik und ihrer Grenzgebiete (3) [Results in Mathematics
  and Related Areas (3)], vol.~14, Springer-Verlag, Berlin, 1988. \MR{MR932724
  (90d:57039)}

\bibitem{KT}
Christian Kassel and Vladimir Turaev, \emph{Braid groups}, Graduate Texts in
  Mathematics, vol. 247, Springer, New York, 2008, With the graphical
  assistance of Olivier Dodane. \MR{2435235 (2009e:20082)}

\bibitem{Mag-74}
Wilhelm Magnus, \emph{Braid groups: a survey}, Proceedings of the {S}econd
  {I}nternational {C}onference on the {T}heory of {G}roups ({A}ustralian {N}at.
  {U}niv., {C}anberra, 1973) (Berlin), Springer, 1974, pp.~463--487. Lecture
  Notes in Math., Vol. 372. \MR{0353290 (50 \#5774)}

\bibitem{MS-00}
Daniel Matei and Alexander~I. Suciu, \emph{Homotopy types of complements of
  {$2$}-arrangements in {${\bf R}^4$}}, Topology \textbf{39} (2000), no.~1,
  61--88. \MR{1710992 (2000h:55028)}

\bibitem{MW-Pascal}
Matthew Miller and Max Wakefield, \emph{Formality of {P}ascal arrangements},
  Proc. Amer. Math. Soc. \textbf{139} (2011), no.~12, 4461--4466. \MR{2823091}

\bibitem{M-k-equal}
Matthew~S. Miller, \emph{Massey products and {$k$}-equal manifolds}, Int. Math.
  Res. Not. IMRN (2012), no.~8, 1805--1821. \MR{2920831}

\bibitem{MW-Edge}
Matthew~S. Miller and Max Wakefield, \emph{Edge colored hypergraphic
  arrangements}, Pure Appl. Math. Q. \textbf{8} (2012), no.~3, 757--779.
  \MR{2900158}

\bibitem{M-78}
John~W. Morgan, \emph{The algebraic topology of smooth algebraic varieties},
  Inst. Hautes {\'E}tudes Sci. Publ. Math. (1978), no.~48, 137--204. \MR{516917
  (80e:55020)}

\bibitem{OT}
Peter Orlik and Hiroaki Terao, \emph{Arrangements of hyperplanes}, Grundlehren
  der Mathematischen Wissenschaften [Fundamental Principles of Mathematical
  Sciences], vol. 300, Springer-Verlag, Berlin, 1992.

\bibitem{PRW-99}
Irena Peeva, Vic Reiner, and Volkmar Welker, \emph{Cohomology of real diagonal
  subspace arrangements via resolutions}, Compositio Math. \textbf{117} (1999),
  no.~1, 99--115. \MR{1693007 (2001c:13021)}

\bibitem{T-arrow}
Hiroaki Terao, \emph{Chambers of arrangements of hyperplanes and {A}rrow's
  impossibility theorem}, Adv. Math. \textbf{217} (2007), 366--378.

\bibitem{Yuz-SmalRat}
Sergey Yuzvinsky, \emph{Rational model of subspace complement on atomic
  complex}, Publ. Inst. Math. (Beograd) (N.S.) \textbf{66(80)} (1999),
  157--164, Geometric combinatorics (Kotor, 1998). \MR{MR1765044 (2002b:52026)}

\bibitem{Yuz-Small}
\bysame, \emph{Small rational model of subspace complement}, Trans. Amer. Math.
  Soc. \textbf{354} (2002), no.~5, 1921--1945 (electronic). \MR{MR1881024
  (2003a:52030)}

\bibitem{ZZ93}
G{\"u}nter~M. Ziegler and Rade~T. {\v{Z}}ivaljevi{\'c}, \emph{Homotopy types of
  subspace arrangements via diagrams of spaces}, Math. Ann. \textbf{295}
  (1993), no.~3, 527--548. \MR{MR1204836 (94c:55018)}

\end{thebibliography}

\end{document}